\documentclass[12pt, a4]{amsart}
\usepackage{amscd, amsmath, amssymb}
\usepackage{fullpage}

\theoremstyle{plain}
\newtheorem{thm}{Theorem}[section]
\newtheorem{lem}[thm]{Lemma}
\newtheorem{pro}[thm]{Proposition}
\theoremstyle{definition}

\newtheorem{ex}[thm]{Example}

\numberwithin{equation}{section}

\begin{document}

\title[Solvability of integral equations with functional terms]{Nontrivial solutions of systems of perturbed Hammerstein integral equations with functional terms} 

\date{}

\author[G. Infante]{Gennaro Infante}
\address{Gennaro Infante, Dipartimento di Matematica e Informatica, Universit\`{a} della
Calabria, 87036 Arcavacata di Rende, Cosenza, Italy}%
\email{gennaro.infante@unical.it}%

\begin{abstract} 
We discuss the solvability of a fairly general class of systems of perturbed Hammerstein integral equations with functional terms that depend on several parameters. The nonlinearities and the functionals are allowed to depend on the components of the system and their derivatives. The results are applicable to systems of nonlocal second order ordinary differential equations subject to functional boundary conditions, this is illustrated in an example. Our approach is based on the classical fixed point index.
\end{abstract}




\maketitle

\section{Introduction}
Nonlocal differential equations have seen recently growing attention by researchers, both in the context of ODEs and PDEs. One motivation for studying this class of equations is that nonlocal terms often occur in physical models, we refer the reader to the paper by Stan\'czy~\cite{Stanczy} for nonlocalities involving averaging processes, and to the review by Ma~\cite{ToFu} for Kirchhoff-type problems.

In the context of ODEs and radial solutions of PDEs in annular domains, a recent and very interesting paper is the one by Goodrich~\cite{Chris20}. 
Goodrich studied the existence of one \emph{positive} solution of the nonlocal
ODE
\begin{equation}\label{nlocalChris-eq}
-A\Bigl(\int_{0}^{1}|u(s)|^{q}\,ds\Bigr) u''(t)=\lambda f(t, u(t)), \ t\in (0,1),
\end{equation}
where $q\geq 1$ and $\lambda$ is a parameter, subject to the Dirichlet boundary conditions (BCs)
\begin{equation}\label{nlocalChrisbc}
u(0)=u(1).
\end{equation}
The approach in~\cite{Chris20} relies on classical fixed point index theory applied in
the cone of positive continuous functions
\begin{equation}\label{Chriscone}
\hat{K}:=\Bigl\{w\in C[0,1]: \int_{0}^{1}w(s)\,ds\geq \hat{c}_1\| w\|_{\infty},\,\,\min_{t\in [a,b]}w(t)\geq \hat{c}_2 \| w\|_{\infty}, w\geq 0\Bigr \},
\end{equation}%
where  $\|w\|_{\infty}:=\sup_{t\in [0,1]}|w(t)|$. Goodrich also studied in~\cite{Chris20} the following generalization of~\eqref{nlocalChris-eq}, namely
\begin{equation*}
-\Bigl(A\Bigl(\int_{0}^{1}|u(s)|^{q}\,ds\Bigr) \Bigr)^{\beta}u''(t)=\lambda (f(t, u(t)))^{\alpha}, \ t\in (0,1),
\end{equation*}
where $\alpha, \beta$ are positive constants, subject to~\eqref{nlocalChrisbc}.

Here we proceed in a different way; rather than studying a specific boundary value problem~(BVP),  we provide new results regarding the existence and non-existence of~\emph{non-zero} solutions of the following class of systems of integral equations with functional terms, namely 
\begin{equation}\label{per-sys-intro}
u_{i}(t)=\lambda_{i} \int_0^1 k_{i}(t,s)f_{i}(s, u(s), u'(s), w_i [u])\,ds + \sum_{j=1,2}\eta_{ij}{\gamma_{ij}}(t) h_{ij}[u],\ t\in [0,1], 
\end{equation}
where $i=1,2,\ldots,n$, $u=(u_{1},\ldots,  u_{n})$, $u'=(u_{1}',\ldots,  u_{n}')$, $f_i$ are continuous, $\gamma_{ij}$ are continuously differentiable, $h_{ij}$ and $w_i $ are suitable functionals, $\lambda_{i}$ and $\eta_{ij}$ are positive parameters.  

When dealing with systems of second order BVPs, the functional terms $w_i$ occurring in~\eqref{per-sys-intro} can be used to incorporate the nonlocalities that appear in the differential equations, while the functionals $h_{ij}$ originate directly from the~BCs.  
In the context of positive solutions, the idea of incorporating the nonlocal terms of differential equations within the nolinearities has been exploited in the case of equations by
Fija\l kowski and Przeradzki~\cite{Bogdan} and Engui\c{c}a and Sanchez~\cite{Luis}, while the case of systems of second order elliptic operators has been considered by the author~\cite{gi-jepeq, gi-BK}. 
We seek solutions of the system~\eqref{per-sys-intro} in a product of cones of a kind that differs from~\eqref{Chriscone}; in particular we work on products of cones in the space $C^1[0,1]$ where the functions are positive on a subinterval of $[0,1]$ and are allowed to change sign elsewhere, this follows the line of research initiated by the author and Webb in~\cite{gijwjiea}. We stress that ours is a larger cone than the one used by the author and Minh\'{o}s~\cite{gi-fm-17}, where some additional constrains on the growth of the derivatives are embedded within the cone, a setting not applicable to the present class of systems due to the assumptions on the kernels. As in the case of elliptic equations~\cite{gi-BK}, our approach can cover different kinds of nonlocalities in the differential equations 
and several types of BCs: local, nonlocal, linear and nonlinear. There exists a wide literature on nonlocal/nonlinear BCs, we refer the reader to the papers~\cite{genupa, Chris-bj} and references therein.

The proof of the existence result relies on the classical fixed point index, while for the non-existence we use an elementary argument. We conclude by illustrating, in an example, how our theoretical results can be applied to a system of nonlocal second order ODEs that presents coupling between the components of the system in the nonlocal terms occurring in the equations and in the BCs.

\section{Existence and nonexistence of nontrivial solutions}
We discuss the solvability of the system of perturbed integral equations of the type
\begin{equation}\label{perhamm-sys-intro}
u_{i}(t)=F_i(u)(t) + \sum_{j=1,2}\eta_{ij}{\gamma_{ij}}(t) h_{ij}[u],\ t\in [0,1], \ i=1,2,\ldots,n,
\end{equation}
where $$F_i(u)(t):=\lambda_{i} \int_0^1 k_{i}(t,s)f_{i}(s, u(s), u'(s), w_i [u])\,ds,$$
$u=(u_{1},\ldots,  u_{n})$, $u'=(u_{1}',\ldots,  u_{n}')$. We make the following assumptions on the terms that occur in~\eqref{perhamm-sys-intro}.
\begin{itemize}
\item[$(C_1)$] For every $i=1,\ldots,n$,  $k_i:[0,1] \times[0,1]\rightarrow \mathbb{R}$ is 
measurable in $s$ for every $t$ and continuous in $t$ for almost every  (a.e.)~$s$,
that is, for every $\tau\in [0,1] $ we have
\begin{equation*}
\lim_{t \to \tau} |k_i(t,s)-k_i(\tau,s)|=0 \ \text{for a.e.}\ s \in [0,1].
\end{equation*}
\item [$(C_2)$]
There exist a subinterval $[a_i,b_i] \subseteq [0,1]$, a constant $\tilde{c}_i=\tilde{c}_i(a_i,b_i) \in (0,1]$ and a function $\Phi_{i0} \in L^{1}(0,1)$ such that
$|k_i(t,s)|\leq \Phi_{i0}(s)$ for $t \in [0,1]$ and a.e.~$s\in [0,1]$ and
$$
k_i(t,s) \geq \tilde{c}_i\Phi_{i0}(s) \text{ for } t\in [a_i,b_i] \text{ and a.\,e. }s \in [0,1].
$$
\item[$(C_3)$] For every $i=1,\ldots,n$, $\dfrac{\partial k_{i}}{\partial t}$ is measurable in $s$ for every $t$, continuous in $t$ except possibly at the point $t=s$ where there can be a jump discontinuity, that is, right and left limits both exist, 
and there exists $\Phi_{i1}(s) \in L^1(0, 1)$ such that $\Bigl|\dfrac{\partial k_{i}}{\partial t}(t,s)\Bigr|\le \Phi_{i1}(s)$ for $t \in [0,1]$ and a.e. $s \in [0,1]$.

\item[$(C_4)$] For every $i=1,\ldots,n$, $f_i:[0,1]\times \mathbb{R}^{2n}\times [0,+\infty)\to [0,+\infty)$ is continuous.

\item[$(C_5)$] For every $i=1,\ldots,n$ and $j=1,2$, we have $\gamma_{ij} \in C^{1} [0,1] $ and there exists a constant $c_{ij}=c_{ij}(a_i,b_i)\in (0,1]$ such that $\gamma_{ij}(t)\geq c_{ij} \|\gamma_{ij}\|_{\infty}$  for every $t\in [a_i,b_i]$.

\item[$(C_6)$] For every $i=1,\ldots,n$ and $j=1,2$, we have $\lambda_i, \eta_{ij},  \in [0,+\infty)$.
\end{itemize}
We work in the product space 
$ \displaystyle\prod_{i=1}^n C^{1}[0,1]$ endowed with the norm 
$$\|u\|:=\max_{i=1,\ldots,n}\{\|u_i\|_{C^{1}}\},$$
where $\|u_i\|_{C^{1}}:=\displaystyle \max \{ \|u_i\|_\infty, \|u_i'\|_\infty \}$. 
We recall that a cone $\mathfrak{C}$ of a real Banach space $X$ is a closed set with $\mathfrak{C}+\mathfrak{C}\subset \mathfrak{C}$, $\mu \mathfrak{C}\subset \mathfrak{C}$ for all $\mu\ge 0$ and $\mathfrak{C}\cap(-\mathfrak{C})=\{0\}$. Here we utilize the cone $K \subset \displaystyle\prod_{i=1}^n C^{1}[0,1]$ defined by 
$$K:=\bigl \{u\in \prod_{i=1}^n \tilde{K}_i \bigr \},$$
where
\begin{equation*}
\tilde{K_{i}}:=\{w\in C^{1}[0,1]:\,\,\min_{t\in [a_{i},b_{i}]}w(t)\geq
{c_{i}}\| w\|_{\infty}\},
\end{equation*}%
here $c_{i}=\min \{\tilde{c}_{i},c_{i1},c_{i2}\}$. Note that $K\neq \{0\}$ since $\hat{1}\in K$, here  $\hat{1}$ denotes the function with each component constant and equal to 1 for every $t\in [0,1]$.
We require  the nonlinear functionals $h_{ij}$ and $w_i$ to act positively on the cone $K$ and to be compact, that is:
\begin{itemize}
\item[$(C_7)$] For every $i=1,\ldots,n$ and $j=1,2$, $h_{ij}: K \to [0,+\infty)$ is continuous and maps bounded sets into bounded sets.
\end{itemize}

\begin{itemize}
\item[$(C_8)$] For every $i=1,\ldots,n,$ $w_i: K \to [0,+\infty)$ is continuous and maps bounded sets into bounded sets.
\end{itemize}
We define the operator $T$ as
\begin{equation*}
T u:=\bigl(T_i u\bigr)_{i=1\ldots n},
\end{equation*}
where
\begin{equation*}
T_i(u)(t)=F_i(u)(t) + \sum_{j=1,2}\eta_{ij}{\gamma_{ij}}(t) h_{ij}[u],\ t\in [0,1], \ i=1,2,\ldots,n.
\end{equation*}
With the assumptions above, it routine to show that $T$ maps $K$ to $K$ and is compact.

The next result summarizes the main properties of the classical fixed point index for compact maps, for more details we refer the reader to~\cite{Amann-rev, guolak}. In what follows the closure and the boundary of subsets of a cone $K$ are understood to be relative to $K$.

\begin{pro}
Let $X$ be a real Banach space and let $\mathfrak{C}\subset X$ be a cone. Let $D$ be an open bounded set of $X$ with $0\in D_{\mathfrak{C}}$ and
$\overline{D}_{\mathfrak{C}}\ne \mathfrak{C}$, where $D_{\mathfrak{C}}=D\cap \mathfrak{C}$.
Assume that $\mathfrak{T}:\overline{D}_{\mathfrak{C}}\to \mathfrak{C}$ is a compact operator such that
$x\neq \mathfrak{T}x$ for $x\in \partial D_{\mathfrak{C}}$. Then the fixed point index
 $i_{\mathfrak{C}}(\mathfrak{T}, D_{\mathfrak{C}})$ has the following properties:

\begin{itemize}

\item[$(i)$] If there exists $e\in \mathfrak{C}\setminus \{0\}$
such that $x\neq \mathfrak{T}x+\lambda e$ for all $x\in \partial D_{\mathfrak{C}}$ and all
$\lambda>0$, then $i_{\mathfrak{C}}(\mathfrak{T}, D_{\mathfrak{C}})=0$.

\item[$(iii)$] If $\mathfrak{T}x \neq \lambda x$ for all $x\in
\partial D_{\mathfrak{C}}$ and all $\lambda > 1$, then $i_{\mathfrak{C}}(\mathfrak{T}, D_{\mathfrak{C}})=1$.

\item[(iv)] Let $D^{1}$ be open bounded in $X$ such that
$\overline{D^{1}_{\mathfrak{C}}}\subset D_{\mathfrak{C}}$. If $i_{\mathfrak{C}}(\mathfrak{T}, D_{\mathfrak{C}})=1$ and $i_{\mathfrak{C}}(\mathfrak{T},
D_{\mathfrak{C}}^{1})=0$, then $\mathfrak{T}$ has a fixed point in $D_{\mathfrak{C}}\setminus
\overline{D_{\mathfrak{C}}^{1}}$. The same holds if
$i_{\mathfrak{C}}(\mathfrak{T}, D_{\mathfrak{C}})=0$ and $i_{\mathfrak{C}}(\mathfrak{T}, D_{\mathfrak{C}}^{1})=1$.
\end{itemize}
\end{pro}
For $\rho\in (0,\infty)$, we define the set
$$
K_{\rho}:=\{u\in K: \|u\|<\rho\}
$$
and the quantities
$$\underline{w}_{i,\rho}:=\inf_{u\in \partial K_{\rho}} w_{i}[u], \ \overline{w}_{i,\rho}:=\sup_{u\in \partial K_{\rho}} w_{i}[u],\ \underline{h}_{ij,\rho}:=\inf_{u\in \partial K_{\rho}} h_{ij}[u],\ \overline{h}_{ij,\rho}:=\sup_{u\in \partial K_{\rho}} h_{ij}[u].
$$

\begin{lem} \label{ind1}
Assume that $u\neq Tu$ on $\partial K_{\rho}$ and suppose that
\begin{enumerate}
\item[$(\mathrm{I}^{1}_{\rho})$]
there exists
 $\rho>0$, such that  
\begin{equation}\label{eqin1}
 \max_{\substack{i=1,\ldots, n\\ l=0,1}}\Bigl\{  \frac{\lambda_{i}\overline{f}_{i,\rho}}{m_{il}}+\sum_{j=1,2}\eta_{ij}\|{\gamma_{ij}^{(l)}}\|_{\infty}\overline{h}_{ij,\rho} \Bigr \}\leq \rho,
\end{equation}
 where 
 $$
\overline{f}_{i,\rho}:=\max_{I_{i,\rho}}f_i(t, x_{1},\ldots, x_{2n},w),\ I_{i,\rho}:=[0,1]\times  [-\rho,\rho]^{2n}\times [\underline{w}_{i,\rho}, \overline{w}_{i,\rho}], $$
$$\frac{1}{m_{il}}:=\begin{cases}\displaystyle \sup_{t\in[0,1]}\int_0^1 |k_i(t,s)|\,ds,\ l=0,\\
\displaystyle \sup_{t\in[0,1]}\int_0^1 \Bigl|\frac{\partial k_{i}}{\partial t}(t,s)\Bigr|\,ds,\ l=1. 
\end{cases}.
$$
\end{enumerate}
Then we have $i_{K}(T, K_{\rho})=1$.
\end{lem}
 \begin{proof}
We prove that 
 $
\sigma  u\neq Tu\ \text{for every}\ u\in \partial K_{\rho}\
\text{and every}\  \sigma >1.
$
If this does not hold, then there exist $u\in \partial K_{\rho}$ and $\sigma >1$ such that $\sigma  u= Tu$. 
Note that if $\| u\| = \rho$ then there exist $i_0 \in \{1,\dots n\}$ such that either 
$\| u_{i_0}\|_{\infty} = \rho$ or $\| u_{i_0}'\|_{\infty} = \rho$. 

We show the case $\| u_{i_0}'\|_{\infty} = \rho$, the case $\| u_{i_0}\|_{\infty} = \rho$ can be treated with a similar argument.
For $t\in [0,1]$ we have
\begin{equation}\label{ineq1}
\begin{aligned}
\sigma u_{i_0}'(t)=&\lambda_{i_0} \int_0^1  \frac{\partial k_{i_0}}{\partial t}(t,s) f_{i_0} (s, u(s), u'(s), w_{i_0} [u])\,ds+ \sum_{j=1,2}\eta_{i_0j}{\gamma_{i_0j}'}(t) h_{i_0j}[u].
\end{aligned}
\end{equation}
From~\eqref{ineq1} we obtain, for $t\in [0,1]$,
\begin{multline}\label{ineq2}
\sigma |u_{i_0}'(t)|\leq \lambda_{i_0} \int_0^1 \Bigl| \frac{\partial k_{i_0}}{\partial t}(t,s)\Bigr| f_{i_0} (s, u(s), u'(s), w_{i_0} [u])\,ds+\sum_{j=1,2}\eta_{i_0j}|{\gamma_{i_0j}'}(t)| h_{i_0j}[u]\\ \leq \lambda_{i_0}\overline{f}_{i_0,\rho} \frac{1}{m_{{i_0}1}}+\sum_{j=1,2}\eta_{i_0j}\|{\gamma_{i_0j}'}\|_{\infty}\overline{h}_{i_0j,\rho}\leq \rho.
\end{multline}
Taking in~\eqref{ineq2} the supremum for $t\in [0,1]$ yields  $\sigma \leq 1$, a contradiction. Therefore we obtain $i_{K}(T, K_{\rho})=1.$
 \end{proof}

\begin{lem} \label{ind01}
Assume that $u\neq Tu$ on $\partial K_{\rho}$ and that
\begin{enumerate}
\item[$(\mathrm{I}^{0}_{\rho})$]
there exists
 $\rho>0$ and $\tilde{\delta_{i}},\delta_{i1},\delta_{i2} \in [0,+\infty),$ such that for every $i\in \{1,\dots, n\}$ we have
 $$
 f_i(t, x_{1},\ldots, x_{2n},w)\geq \tilde{\delta_{i}} x_{i},\ \text{on}\ [a_i,b_i]\times \displaystyle\prod_{j=1}^{2n} [\theta_j\rho,\rho]\times [\underline{w}_{i,\rho}, \overline{w}_{i,\rho}],
 $$
 $$
{h}_{ij,\rho}[u]\geq \delta_{ij}\|u_i\|_{\infty},\ \text{for every}\ u \in \partial K_{\rho}\ \text{and}\  j=1,2,
 $$
 and 
\begin{equation}
 \min_{\substack{i=1,\ldots, n}}\Bigl\{ \lambda_{i} \tilde{\delta_{i}}\tilde{c}_{i}\frac{1}{M_{i}} 
+ \sum_{j=1,2}\eta_{ij}c_{ij} \delta_{ij}\|\gamma_{ij}\|_{\infty} \Bigr \}\geq 1,
\end{equation}
 where 
 $$
\frac{1}{M_{i}}:= \inf_{t\in[a_i,b_i]}\int_{a_i}^{a_i} k_i(t,s)\,ds,\quad \theta_{j}:=\begin{cases} 
-1, &j\neq i,\\
0,& j= i.\\
\end{cases}
$$
\end{enumerate}
Then we have $i_{K}(T, K_{\rho})=0$.
\end{lem}
 \begin{proof}
We show that 
$
u\neq Tu+\sigma \hat{1}$ for every $u\in \partial K_{\rho}$ and every $\sigma  >0.$
If not, there exists $u\in \partial K_{\rho}$ and $\sigma  >0$ such that
$
u= Tu+\sigma \hat{1}. 
$
Then there exist $i_0\in \{1,\ldots, n\}$ and $\tilde{\rho}$ such that $0< \tilde{\rho}=\| u_{i_0} \|_{\infty}\leq \|u\|=\rho $. 
For every $t\in [a_i,b_i]$ we have
\begin{multline*}
\tilde{\rho}\geq u_{i_0}(t)=\lambda_{i_0} \int_0^1 k_{i_0} (t,s) f_{i_0} (s, u(s), u'(s), w_{i_0} [u])\,ds
+ \sum_{j=1,2}\eta_{i_0j}\gamma_{i_0j}(t) h_{i_0j}[u] + \sigma\\   
\ge  \lambda_{i_0} \int_{a_i}^{b_i} k_{i_0} (t,s) \tilde{\delta}_{i_0} u_{i_0}(s)\,ds + \sum_{j=1,2}\eta_{i_0j}\gamma_{i_0j}(t) \delta_{i_0j}\|u_{i_0}\|_{\infty}  + \sigma\\
\ge  \lambda_{i_0}\tilde{\delta}_{i_0} \tilde{c}_{i_0}\tilde{\rho} \frac{1}{M_{i_0}} 
+ \sum_{j=1,2}\eta_{i_0j}c_{i_0j} \delta_{i_0j}\|\gamma_{i_0j}\|_{\infty}\tilde{\rho}+\sigma \geq  \tilde{\rho} +\sigma,
\end{multline*}
a contradiction, since $\sigma  >0$. Therefore we obtain $i_{K}(T, K_{\rho})=0.$
 \end{proof}
In the next Lemma, we restrict the growth of the nonlinearities in only one of the components.
\begin{lem} \label{ind02}
Assume that $u\neq Tu$ on $\partial K_{\rho}$ and that
\begin{enumerate}
\item[$(\mathrm{I}^{0}_{\rho})^{\star}$]
there exist
 $\rho>0$ and $i_{0}\in \{1,\dots, n\}$ such that
\begin{equation}\label{idxstar}
 \frac{\lambda_{i_0}\underline{f}_{i_0,\rho}}{M_{i_0}} 
+ \sum_{j=1,2}\eta_{i_0j}c_{i_0j}\|\gamma_{i_0j}\|_{\infty} \underline{h}_{{i_0}j,\rho}\geq \rho,
\end{equation}
 where 
 $$
\underline{f}_{i,\rho}:=\min_{J_{i,\rho}}f_i(t, x_{1},\ldots, x_{2n},w),\
J_{i,\rho}:=[a_i,b_i]\times \displaystyle\prod_{j=1}^{2n} [\theta_j\rho,\rho]\times [\underline{w}_{i,\rho}, \overline{w}_{i,\rho}].
$$
\end{enumerate}
Then we have $i_{K}(T, K_{\rho})=0$.
\end{lem}
 \begin{proof}
We show that 
$
u\neq Tu+\sigma \hat{1}$ for every $u\in \partial K_{\rho}$ and every $\sigma  >0.$
If not, there exists $u\in \partial K_{\rho}$ and $\sigma  >0$ such that
$
u= Tu+\sigma \hat{1}. 
$
Note that $\| u_{i_0} \|_{\infty}\leq \|u\|=\rho$, therefore
for every $t\in [a_i,b_i]$ we have
\begin{multline*}
\rho \geq u_{i_0}(t)=\lambda_{i_0} \int_0^1 k_{i_0} (t,s) f_{i_0} (s, u(s), u'(s), w_{i_0} [u])\,ds
+ \sum_{j=1,2}\eta_{i_0j}\gamma_{i_0j}(t) h_{i_0j}[u] + \sigma\\   
\ge  \lambda_{i_0} \int_{a_i}^{b_i} k_{i_0} (t,s) \underline{f}_{i_0,\rho} \,ds + \sum_{j=1,2}\eta_{i_0j}\gamma_{i_0j}(t) \underline{h}_{{i_0}j,\rho} + \sigma\\
\ge  \lambda_{i_0}\underline{f}_{i_0,\rho}\frac{1}{M_{i_0}} 
+ \sum_{j=1,2}\eta_{i_0j}c_{i_0j} \underline{h}_{{i_0}j,\rho}+\sigma\geq \rho +\sigma,
\end{multline*}
a contradiction, since $\sigma  >0$. Therefore we obtain $i_{K}(T, K_{\rho})=0.$
 \end{proof}
With these ingredients we can state the following existence and localization result.
\begin{thm}\label{thmsol}
Assume that either of the following conditions holds.
\begin{enumerate}
\item[$(S)$]  There exist $\rho _{1},\rho _{2}\in (0, +\infty )$ with $\rho
_{1}<\rho _{2}$ such that
 $(\mathrm{I}_{\rho _{1}}^{0})$ and $(\mathrm{I}_{\rho _{2}}^{1})$ are satisfied.
\item[$(S)^{\star}$] There exist $\rho _{1},\rho _{2}\in (0, +\infty )$ with $\rho
_{1}<\rho _{2}$ such that $(\mathrm{I}_{\rho _{1}}^{0})^{\star}$ and $(\mathrm{I}_{\rho _{2}}^{1})$ are satisfied.
\end{enumerate}
Then the 
system~\eqref{perhamm-sys-intro} has at least one solution $u\in K$, with $\rho_1\leq \|u\|\leq \rho_2$.
\end{thm}

\begin{proof}
We prove the result under the assumption $(S)$, the other case is similar.
If $T$ has fixed point either on $\partial {K_{\rho _{1}}}$ or on $\partial {K_{\rho _{2}}}$ we are done. If this is not the case, by Lemma~\ref{ind01} we have $i_{K}(T, K_{\rho_1})=0$ and by Lemma~\ref{ind1} we obtain $i_{K}(T, K_{\rho_2})=1$. Therefore we have $i_{K}(T, K_{\rho_2}\setminus \overline{K}_{\rho_1})=1,$ which proves the result.
\end{proof}

We now provide a non-existence result that allows different growths in the components of the system.
\begin{thm}\label{nonexthm}
Let $\mathcal{I},\mathcal{J}\subset \{1,\ldots,n\}$ be such that $\mathcal{I}\cap \mathcal{J}=\emptyset$ and $\mathcal{I}\cup  \mathcal{J}=  \{1,\ldots,n\}$
and assume that there exists $\rho>0$ such that the following conditions are satisfied:
\begin{enumerate}
\item[$(N_{\mathcal{I}})$]
There exist $\tilde{\xi}_{i}, \xi_{i1}, \xi_{i2}\in [0,+\infty)$ such that, for every $i\in \mathcal{I}$ we have
\begin{equation*}
f_{i} (t, x_{1},\ldots, x_{2n},w)\leq \tilde{\xi}_{i} |x_{i}|,\ \text{on}\ [0,1]\times  [-\rho,\rho]^{2n}\times 
\Bigl[\inf_{u\in \overline{K}_{\rho}} w_{i}[u], \sup_{u\in \overline{K}_{\rho}} w_{i}[u]\Bigr],
\end{equation*}
\begin{equation*}
h_{ij}[u]\leq \xi_{ij}\|u_i\|_{\infty}, \ \text{for every}\ u \in \overline{K}_{\rho}\ \text{and}\  j=1,2,
\end{equation*}
\begin{equation}\label{nonexineq}
\max_{i\in \mathcal{I}}\Bigl\{ \frac{\lambda_{i}\tilde{\xi}_{i}}{m_{i0}}
+ \sum_{j=1,2}\eta_{ij}\xi_{ij}\|\gamma_{ij}\|_{\infty}\Bigr\}<1.
 \end{equation}
\item[$(N_{\mathcal{J}})$] There exist $\tilde{\delta}_{i},\delta_{i1},\delta_{i2} \in [0,+\infty),$ such that for every $i\in \mathcal{J}$ we have
 $$
 f_i(t, x_{1},\ldots, x_{2n},w)\geq \tilde{\delta_{i}} x_{i},\ \text{on}\quad [a_i,b_i]\times \displaystyle\prod_{j=1}^{2n} [\theta_j\rho,\rho]\times \Bigl[\inf_{u\in \overline{K}_{\rho}} w_{i}[u], \sup_{u\in \overline{K}_{\rho}} w_{i}[u]\Bigr],
 $$
 $$
h_{ij}[u]\geq \delta_{ij}\|u_i\|_{\infty},\ \text{for every}\ u \in \overline{K}_{\rho}\ \text{and}\  j=1,2,
 $$
 and 
\begin{equation}\label{nonexineq2}
 \min_{\substack{i\in \mathcal{J}}}\Bigl\{\frac{ \lambda_{i} \tilde{\delta_{i}}\tilde{c}_{i}}{M_{i}} 
+ \sum_{j=1,2}\eta_{ij}c_{ij} \delta_{ij}\|\gamma_{ij}\|_{\infty} \Bigr \}> 1,
\end{equation}
\end{enumerate}
Then the system~\eqref{perhamm-sys-intro} has at most the zero solution in $\overline{K}_{\rho}$. 
\end{thm}
\begin{proof}
Assume that there exist $u\in \overline{K}_{\rho}\setminus \{0\}$ such that $Tu=u$. 
Then there exists $i_{0}\in \{1,\ldots, n\}$ and $\tilde{\rho}\in (0,\rho]$ such that $\|u_{i_{0}}\|_{\infty}=\tilde{\rho}$.

If $i_{0}\in \mathcal{I}$, then, by means of the assumptions in $(N_{\mathcal{I}})$, for every $t\in [0,1]$ we have
\begin{multline}\label{nonex-est}
|u_{i_0}(t)|\leq \lambda_{i_0} \int_0^1 |k_{i_0} (t,s)| f_{i_0} (s, u(s), u'(s), w_{i_0} [u])\,ds
+ \sum_{j=1,2}\eta_{i_0j}|\gamma_{i_0j}(t)| h_{i_0j}[u]\\   
\leq  \lambda_{i_0} \int_0^1 |k_{i_0} (t,s)| \tilde{\xi}_{i_0}|u_{i_0}(s)|\,ds
+ \sum_{j=1,2}\eta_{i_0j}|\gamma_{i_0j}(t)| \xi_{i_0j}\|u_{i_0}\|_{\infty}\\  
\leq \Bigl( \lambda_{i_0}\tilde{\xi}_{i_0} \frac{1}{m_{i_00}}
+ \sum_{j=1,2}\eta_{i_0j}\xi_{i_0j}\|\gamma_{i_0j}\|_{\infty} \Bigr )\tilde{\rho}<\tilde{\rho}.
\end{multline}
Passing to the supremum for $t\in [0,1]$ in~\eqref{nonex-est} gives $\tilde{\rho} < \tilde{\rho}$, a contradiction.

If $i_{0}\in \mathcal{J}$, then the assumptions in $(N_{\mathcal{J}})$ imply that, for every $t\in [a_i,b_i]$, we have 
\begin{multline*}
\tilde{\rho}\geq u_{i_0}(t)=\lambda_{i_0} \int_0^1 k_{i_0} (t,s) f_{i_0} (s, u(s), u'(s), w_{i_0} [u])\,ds
+ \sum_{j=1,2}\eta_{i_0j}\gamma_{i_0j}(t) h_{i_0j}[u] \\   
\ge  \lambda_{i_0} \int_{a_i}^{b_i} k_{i_0} (t,s) \tilde{\delta}_{i_0} u_{i_0}(s)\,ds + \sum_{j=1,2}\eta_{i_0j}\gamma_{i_0j}(t) \delta_{i_0j}\|u_{i_0}\|_{\infty}  \\
\ge   \Bigl(\lambda_{i_0}\tilde{\delta}_{i_0} \tilde{c}_{i_0} \frac{1}{M_{i_0}} 
+ \sum_{j=1,2}\eta_{i_0j}c_{i_0j} \delta_{i_0j}\|\gamma_{i_0j}\|_{\infty}\Bigr )\tilde{\rho} >  \tilde{\rho},
\end{multline*}
a contradiction.
\end{proof}
We conclude with the following example, which illustrates the applicability of the above results.
\begin{ex}
Consider the system
\begin{equation} \label{exmp}
\left\{
\begin{array}{cl}
-\bigl(e^{u_2(\frac{1}{2})}+\int_{0}^{1} (u_1'(t))^2\,dt\bigr) u''_1(t)=\lambda_1 e^{u_1(t)}(1+  (u_2'(t))^2),& t\in [0,1] , \\
-e^{(\int_{0}^{1}(u_1'(t) + u_2'(t))^2\,dt )}u''_2=\lambda_2 (u_2(t)u_1'(t))^2, & t\in [0,1] , \\
u_1'(0)+\eta_{11} h_{11}[(u_1,u_2)]=0,\ \frac{1}{4} u_1'(1)+u_1(\frac{1}{2})=0,\\
u_2'(0)+\eta_{21} h_{21}[(u_1,u_2)]=0,\ \frac{1}{4} u_2(\frac{1}{2})+u_2(1)=0,
\end{array}%
\right. 
\end{equation}%
where 
$$
h_{11}[(u_1,u_2)]= \bigl( u_{1}'(\frac{1}{2})\bigr)^{2}+\bigl( u_{2}'(\frac{1}{2})\bigr)^{2}\ \mbox{and} \ h_{21}[(u_1,u_2)]=(u_2(\frac{1}{2}))^2 \int_{0}^{1} (u_1'(t))^2\,dt.
$$
We investigate the existence of solutions of the system~\eqref{exmp} with norm less than or equal to~1.
The system~\eqref{exmp} can be written in the form
\begin{equation*}
\left\{
\begin{array}{l}
u_1(t)=\eta_{11}\gamma_{11}(t)h_{11}[(u_1,u_2)]
+\lambda_1\int_{0}^{1}k_1(t,s) f_{1}(s, u_1(s), u_2(s), u_1'(s),u_2'(s), w_1[(u_1,u_2)])\,ds, \\
u_2(t)=\eta_{21}\gamma_{21}(t)h_{21}[(u_1,u_2)]
+\lambda_2\int_{0}^{1}k_2(t,s) f_{2}(s, u_1(s), u_2(s), u_1'(s),u_2'(s), w_1[(u_1,u_2)])\,ds, \\
\end{array}%
\right. 
\end{equation*}%
where
$$
f_{1}(t, u_1(t), u_2(t), u_1'(t),u_2'(t), w_1[(u_1,u_2)]=e^{u_1(t)}(1+  (u_2'(t))^2)w_1[(u_1,u_2)],
$$
$$
f_{2}(t, u_1(t), u_2(t), u_1'(t),u_2'(t), w_1[(u_1,u_2)]=(u_2(t)u_1'(t))^2w_2[(u_1,u_2)],
$$
$$
\gamma_{11}(t)=\frac{3}{4}- t,\ w_1[(u_1,u_2)]=\bigl(e^{u_2(\frac{1}{2})}+\int_{0}^{1} (u_1'(t))^2\,dt\bigr)^{-1},
$$
\begin{equation*}
k_1(t,s)=\frac{1}{4}+\begin{cases} \frac{1}{2}-s,\ &s\leq \frac{1}{2},\\0,\ &s>\frac{1}{2},
\end{cases}
-\begin{cases} t-s,\ &s\leq t,\\ 0,\ &s>t,
\end{cases}
\end{equation*}
$$
\gamma_{21}(t)=\Bigl(\frac{9}{10}-t \Bigr),\  w_2[(u_1,u_2)]=e^{-\int_{0}^{1}(u_1'(t) + u_2'(t))^2\,dt },
$$
\begin{equation*}
k_2(t,s)=\frac{4}{5}(1-s)+\begin{cases}
\frac{1}{5}(\frac{1}{2} -s), &  s \le \frac{1}{2}\\ \quad 0,&
s>\frac{1}{2}
\end{cases}
 - \begin{cases} t-s, &s\le t \\ \quad 0,&s>t.
\end{cases}
\end{equation*}
The kernel $k_1$ is non-negative in for $t\in [0,3/4]$ and can change sign for $t\in [3/4,1]$. The computation of the constants related to 
$k_1$ and $\gamma_{11}$ can be found
for example in \cite{genupa, gijwems} and references therein, and read as follows
$$
\Phi_{10} (s)=\|\gamma_{11}\|_{\infty}=\frac{3}{4}, \frac{1}{m_{10}}=\frac{3}{8}.
$$
The choice of $[a_1,b_1]=[0, \frac{3}{8}]$ yields $c_1=\tilde{c}_1=c_{11}=\frac{1}{3}, \frac{1}{M_1}=\frac{9}{64}$. Furthermore note that
 \begin{equation*}
\gamma_{11}'(t):=-1, \quad 
\frac{\partial k_{1}}{\partial t}(t,s)=
\begin{cases} -1,\ &s< t,\\ 0,\ &s>t,
\end{cases}
\end{equation*}
and thus we obtain
 \begin{equation*}
1=\|\gamma_{11}'\|=\Phi_{11} (s)=\frac{1}{m_{11}}.
\end{equation*}
The kernel $k_2$ is non-negative in for $t\in [0,1/2]$ but can change sign for $t\in [1/2,1]$. 
In this case $\Phi_{20}=(1-s)$ and fixing
$[a_2,b_2]=[0, \frac{1}{2}]$ gives $\tilde{c}_2=\frac{2}{5}$, see~\cite{gijwjiea}. By direct calculation we obtain 
$$
\|\gamma_{21}\|_{\infty}=\frac{9}{10},\ \frac{1}{m_{20}}=\frac{21}{40},\ c_{21}=\frac{4}{9},\ \frac{1}{M_2}=\frac{2}{5},
$$
thus we have $c_2=\tilde{c}_2=\frac{2}{5}$. Reasoning as above yields
 \begin{equation*}
1=\|\gamma_{21}'\|=\Phi_{21} (s)=\frac{1}{m_{21}}.
\end{equation*}
Note that for $(u_1,u_2)\in \overline{K}_{1}$ we have $$(1+e)^{-1}\leq w_1[(u_1,u_2)]\leq e\ \text{and}\ e^{-4}\leq w_2[(u_1,u_2)]\leq 1,$$
thus we have $$\overline{f}_{1,1}\leq 2e^2,\ \overline{f}_{2,1}\leq 1,\ \overline{h}_{11,1}\leq 2 ,\ \overline{h}_{21,1}\leq 1. 
$$
Therefore~\eqref{eqin1} is satisfied if the parameters $\lambda_{1}, \eta_{11}, \lambda_{2}, \eta_{21}$ satisfy the restriction
\begin{equation}\label{exr1}
\max \Bigl\{2(e^2\lambda_{1}+\eta_{11}), \lambda_{2}+\eta_{21} \Bigr \}\leq 1.
\end{equation}
Note that 
$
\underline{f}_{1,\rho_0}\geq e^{-\rho_0}(1+e)^{-1}
$
for $0<\rho_0<1$, therefore~\eqref{idxstar} is satified if $\lambda_1>0$ and $\rho_0$ is sufficiently small.

If the coefficients satisfy \eqref{exr1} and $\lambda_1>0$, by Theorem~\ref{thmsol} we obtain a non-zero solution in $K$ with $\|(u_1,u_2)\|\leq 1$; 
this happens for example for $\lambda_{1}=1/20, \eta_{11}=1/10, \lambda_{2}=\eta_{21}=1/2$.

Now fix $\mathcal{J}=\{1\}$ and $\mathcal{I}=\{2\}$. Observe that
$$
 f_1(t, x_{1},\ldots, x_{2n},w)\geq \frac{21}{30} x_{1},\ \text{on}\quad [0,3/4]\times \displaystyle\prod_{j=1}^{2n} [\theta_j,1]\times \bigl[(1+e)^{-1}, e\bigr],
 $$
 $$
h_{11}[u]\geq 0,\ \text{for every}\ u \in \overline{K}_{1},
 $$
thus the inequality~\eqref{nonexineq2} is satisfied for
\begin{equation}\label{ncon1}
 \lambda_{1}> \frac{640}{21}.
\end{equation}
Furthermore note that we have
\begin{equation*}
f_{2} (t, x_{1},\ldots, x_{2n},w)\leq  |x_{2}|,\ \text{on}\ [0,1]\times  [-1,1]^{2n}\times 
\Bigl[e^{-4},  1\Bigr],
\end{equation*}
\begin{equation*}
h_{21}[u]\leq \|u_2\|_{\infty}, \ \text{for every}\ u \in \overline{K}_{1},
\end{equation*}
thus the condition~\eqref{nonexineq} is satified if
\begin{equation}\label{ncon2}
\Bigl\{ \lambda_{2}\frac{21}{40}+
\eta_{21}\frac{9}{10}\Bigr\}<1.
 \end{equation}
Note that the trivial solution does not satisfy the system~\eqref{exmp}; Theorem~\ref{nonexthm} yields that the system~\eqref{exmp} has no solutions in $K$ with norm less than or equal to $1$ whenever \eqref{ncon1} and \eqref{ncon2} are satisfied; this happens for example when 
$\lambda_{1}=31, \eta_{11}=\lambda_{2}=\eta_{21}=1$.
\end{ex}

\section*{Acknowledgements}
G. Infante was partially supported by G.N.A.M.P.A. - INdAM (Italy).

\end{document}